\documentclass[11pt]{article}
\usepackage{amsfonts,amsmath,amssymb,amsthm}
\usepackage[mathscr]{euscript}
\usepackage{mathrsfs,indentfirst,enumitem,color,caption}
\usepackage{graphicx}
\newcommand\bX{{\mathbf X}}
\newcommand\bc{{\mathbf c}}

\newcommand\boo{\bar\oo}

\numberwithin{equation}{section}

\newcommand{\lbl}[1]{\label{#1}}

\newtheorem{theo}{Theorem}[section]

\newtheorem{lem}{Lemma}[section]
\newtheorem{remark}{Remark}[section]
\newtheorem{coro}{Corollary}[section]
\newtheorem{defi}{Definition}[section]
\newcommand{\be}{\begin{equation}}
\newcommand{\ee}{\end{equation}}
\newcommand\bes{\begin{eqnarray}} \newcommand\ees{\end{eqnarray}}
\newcommand{\bess}{\begin{eqnarray*}}
\newcommand{\eess}{\end{eqnarray*}}

\newcommand\kk{\left}
\newcommand\rr{\right}
\newcommand\dd{\displaystyle}

\newcommand\lm{\lambda}

\newcommand\oo{\Omega}
\newcommand\yy{\infty}

\newcommand\qq{\eqref}

\newcommand\sk{\smallskip}

\newcommand\dx{{\rm d}x}
\newcommand\bR{{\mathbb R}}
\begin{document}
\begin{center}{\Large\bf The diffusive eco-epidemiological prey-predator model}\\[2mm] {\Large\bf with infectious diseases in prey}\footnote{This work was
supported by NSFC Grant 12171120.}\\[4mm]
{\Large Mingxin Wang\footnote{{\sl E-mail}: mxwang@hpu.edu.cn}}\\[0.5mm]
 {School of Mathematics and Information Science, Henan Polytechnic University, Jiaozuo 454000, China}
\end{center}

\begin{quote}
\noindent{\bf Abstract}. This paper focus on the diffusive eco-epidemiological prey-predator model with infectious diseases in prey, and with the homogeneous Neumann and Dirichlet boundary conditions, respectively. When boundary conditions are  homogeneous Neumann boundary conditions, we give a complete conclusion about the stabilities of nonnegative constant equilibrium solutions. The results show that such a problem has neither periodic solutions nor Turing patterns. When boundary conditions are homogeneous Dirichlet boundary conditions, we first establish the necessary and sufficient conditions for the existence of positive equilibrium solutions, and prove that the positive equilibrium solution is unique when it exists. Then we study the global asymptotic stabilities of trivial and semi-trivial nonnegative equilibrium solutions.


 \end{quote}

\setlength{\baselineskip}{16pt} \pagestyle{myheadings}
\section{Introduction}{\setlength\arraycolsep{2pt}

The effect of disease in ecological system is an important issue from mathematical and ecological points of view. In mathematical biology, one of the newest branch
in which simultaneously both the ecological and epi-demiological issues are taken under consideration is eco-epidemiology. In the presence of the virus, prey population is divided into two parts: susceptible and infected. Let $S$, $I$ and $P$ be the densities of susceptible prey, infected prey and predator, respectively. Recently, Huang and Wang (\cite{HWang22}) proposed the following eco-epidemiological prey-predator model with infectious diseases in prey and the depletion of food supply by all prey:
 \bess\left\{\begin{array}{ll}
S'=a(S+I)-b S-c(S+I)S-kIS-\ell SP,\;\;&t>0,\\[.5mm]
 I'=k IS-b I-c(S+I)I-\gamma IP,\;\;&t>0,\\[.5mm]
P'=\theta SP+\sigma IP-\rho P,\;\;&t>0,\\[0.5mm]
  S(0)=S_0>0,\;\; I(0)=I_0>0,\;\;P(0)>0.
\end{array}\right.
 \eess
In this system, $a$ and $b$ ($a>b$) are the birth and death rates of prey, respectively, the term  $a(S+I)$ in the first equation means that the susceptible prey $S$ and infected prey $I$ have the same fertility, and newborns are healthy and susceptible to infection; $c=(a-b)/M$, where $M$ is the environmental carrying capacity for prey; the term $c(S+I)$ in each of the first two equations means depletion of the food supply by all prey; $k$ is the infection coefficient, $\rho$ is the death rate of predator; terms $\ell S$ and $\gamma I$ denote the predator's predation rates for the susceptible and infected prey, respectively. If we set  $\theta=\theta'\ell$ and $\sigma=\sigma'\gamma$. Then $\theta'$ and $\sigma'$ can be regarded as conversion rates.

Huang and Wang (\cite{HWang22}) first proved the well-posedness and positivity of solutions, and then investigated the nonnegative equilibrium points and their stabilities.

Taking into account the inhomogeneous distribution of the prey and the predator in different spatial locations within a fixed bounded domain $\Omega$ at any given
time, and the natural tendency of each species to diffuse to areas of smaller population concentration. Let $d$ and $D$ be the diffusion coefficients of prey and predator, respectively. In this paper we are naturally led to the following corresponding reaction-diffusion systems with the homogeneous Neumann boundary conditions
  \bes
 \left\{\begin{array}{lll}
 S_t-\Delta S=a(S+I)-b S-c(S+I)S-kIS-\ell SP,\;\; &x\in\Omega, \; t>0,\\[1mm]
  I_t-d\Delta I=k IS-b I-c(S+I)I-\gamma IP,\ &x\in\Omega, \ t>0,\\[1mm]
 P_t-D\Delta P=\theta SP+\sigma IP-\rho P,\ &x\in\Omega, \ t>0,\\[1mm]
 \dd\frac{\partial S}{\partial\nu}=\frac{\partial I}{\partial\nu}=
 \frac{\partial P}{\partial\nu}=0, \ &x\in\partial\Omega, \ t>0,\\[2mm]
 S(x,0)=S_0(x)>0,\ I(x,0)=I_0(x)>0,\ P(x,0)=P_0(x)>0,\ &x\in\bar\Omega,
 \lbl{1.1}\end{array}\right.
 \ees
and the homogeneous Dirichlet boundary conditions
\bes
 \left\{\begin{array}{lll}
 S_t-d\Delta S=a(S+I)-b S-c(S+I)S-kIS-\ell SP,\;\; &x\in\Omega, \; t>0,\\[1mm]
  I_t-d\Delta I=k IS-b I-c(S+I)I-\gamma IP,\ &x\in\Omega, \ t>0,\\[1mm]
 P_t-D\Delta P=\theta SP+\sigma IP-\rho P,\ &x\in\Omega, \ t>0,\\[1mm]
  S=I=P=0, \ &x\in\partial\Omega, \ t>0,\\[1mm]
 S(x,0)=S_0(x)>0,\ I(x,0)=I_0(x)>0,\ P(x,0)=P_0(x)>0,\ &x\in\bar\Omega,
 \lbl{1.2}\end{array}\right.
 \ees
where $\Omega\subset\mathbb{R}^N$ is a bounded domain with smooth
boundary $\partial\Omega$, and $\nu$ is the outward normal vector of $\partial\Omega$. In the problem \qq{1.1}, $\nu$ is the outward unit normal vector of the boundary $\partial\Omega$, the initial data $S_0, I_0, P_0\in W^2_p(\oo)$ for some $p>1$ and satisfies the compatibility conditions:
 \[\frac{\partial S_0}{\partial\nu}=\frac{\partial I_0}{\partial\nu}=
 \frac{\partial P_0}{\partial\nu}=0,\;\;\;x\in\partial\Omega.\]
In the problem \qq{1.2}, the initial data $S_0, I_0, P_0\in W^2_p(\oo)$ for some $p>1$ and satisfies the compatibility conditions:
 \[S_0=I_0=P_0=0,\;\;\;x\in\partial\Omega.\]

The organization of this paper is as follows. Section 2 is devoted to investigate the well-posedness and positivity of solutions of \qq{1.1} and \qq{1.2}. In the forthcoming sections, we only consider the special case $(\gamma, \sigma)=(\ell,\theta)$. In Section 3, a complete conclusion about the stabilities of nonnegative constant equilibrium solutions of \qq{1.1} is established. The results show that such a problem has neither periodic solution nor Turing patterns. Section 4 is devoted to study the equilibrium problem of \qq{1.2}:
\bes
 \left\{\begin{array}{lll}
-d\Delta S=a(S+I)-b S-c(S+I)S-kSI-\ell S P,\;\; &x\in\Omega, \\[1mm]
 -d\Delta I=k SI-b I-c(S+I)I-\ell IP,\ &x\in\Omega,\\[1mm]
 -D\Delta P=\theta (S+I)P-\rho P,\ &x\in\Omega,\\[1mm]
  S=I=P=0, \ &x\in\partial\Omega.
 \end{array}\right.
 \lbl{1.3}\ees
Obviously, $\boldsymbol{0}=(0,0,0)$ is a trivial solution of \qq{1.3}.

For $q\in L^\infty(\Omega)$, we denote the principal eigenvalue of
 \bess\left\{\begin{array}{lll}
 -d\Delta\phi+q(x)\phi=\lambda\phi,\;\;&x\in\Omega,\\[1mm]
\phi=0,\;\;&x\in\partial\Omega
\end{array}\right.
 \eess
by $\lambda_1^d(q)$. When $q\equiv 0$, we simply denote $\lambda_1^d(0)=\lambda_0^d$. It is well known that
 $$\lambda_1^d(q)=\lambda_0^d+q$$
when $q$ is a constant.

It is well known that the problem
 \bes
 \left\{\begin{array}{lll}
-d\Delta S=(a-b)S-c S^2,\;\; &x\in\Omega, \\[1mm]
 S=0, \ &x\in\partial\Omega
 \lbl{1.4}\end{array}\right.
 \ees
has a positive solution, denoted by $S^*$, if and only if $a-b>\lambda^d_0$, i.e., $\lm_1^d(b-a)<0$, and $S^*$ is unique, non degenerate and globally asymptotically stable  when it exists. Obviously, $(S^*, 0, 0)$ is a semi-trivial nonnegative solution of \qq{1.3}. Besides, the problem \qq{1.3} may have the semi-trivial nonnegative solutions $(\tilde S, \tilde I,0)$ and $(\hat S,0,\hat P)$, where $(\tilde S, \tilde I)$ is a positive solution of
\bes
 \left\{\begin{array}{lll}
-d\Delta S=a(S+I)-b S-c(S+I)S-k SI,\;\; &x\in\Omega, \\[1mm]
 -d\Delta I=k SI-b I-c(S+I)I,\ &x\in\Omega,\\[1mm]
   S=I=0, \ &x\in\partial\Omega,
 \lbl{1.5}\end{array}\right.
 \ees
and $(\hat S,\hat P)$ is a positive solution of
 \bes
 \left\{\begin{array}{lll}
-d\Delta S=(a-b)S-c S^2-\ell S P,\;\; &x\in\Omega, \\[1mm]
  -D\Delta P=\theta S P-\rho P,\ &x\in\Omega,\\[1mm]
  S=P=0, \ &x\in\partial\Omega.
 \lbl{1.6}\end{array}\right.
 \ees
To study the existence of positive solutions of \qq{1.3}, \qq{1.5} and \qq{1.6}, in Subsection 4.1 we state some abstract results about topological degree in cones and give the estimates of positive solutions of \qq{1.3}. Problems \qq{1.5} and \qq{1.6} are investigated in Subsections 4.2 and 4.3, respectively. The necessary and sufficient conditions for the existence of positive solution, and the uniqueness and non degeneracy of positive solutions are established. For the problem \qq{1.3}, the necessary and sufficient conditions for the existence of positive solutions, and the uniqueness of positive solutions are obtained in Section 4.4. The global asymptotic stabilities of trivial and semi-trivial nonnegative equilibrium solutions are given in Section 4.5.

\section{The well-posedness and positivity of solutions for the problems \qq{1.1} and \qq{1.2}}
\setcounter{equation}{0} {\setlength\arraycolsep{2pt}

\begin{theo}\lbl{t2.1} The problem \eqref{1.1} has a unique global positive solution $(S, I, P)$ and there exists a positive constant $C$ depends only on the parameters appearing in the differential equations of \qq{1.1} and the maximum norms of the initial data $S_0(x)$, $I_0(x)$ and $P_0(x)$ such that
 \bes
 0< S,\; I\le\max\kk\{\dd\max_{\bar\oo}\{S_0(x)+I_0(x)\},\,\dd\frac{a-b}c\rr\},
\;\;\; 0< P\le C\;\;\;\text{for}\;\;x\in\bar\oo,\; t>0.
 \lbl{2.1}\ees
Moreover, for any given $0<\alpha<1$ and $\tau>0$, $S,\,I,\,P\in C^{2+\alpha,\,1+\alpha/2}(\overline\Omega\times[\tau, \infty))$, and there exists a constant $C(\tau)$ such that
\begin{eqnarray}
\|S,\,I,\,P\|_{C^{2+\alpha,\,1+\alpha/2}(\overline\Omega\times[\tau, \infty))}\leq C(\tau).
\label{2.2}\end{eqnarray}
\end{theo}

\begin{proof} Firstly, it is standard to prove the local existence and uniqueness of the solution $(S, I, P)$ of \eqref{1.1}. Let $T>0$ be the maximum existence time of  $(S, I, P)$.

{\it Step 1}. Making use of the differential equation and boundary condition of $I$ and the initial value $I_0(x)>0$, it is easy to see that $I(x,t)>0$ for $x\in\bar\oo$ and $0<t<T$ by the strong maximum principle. Therefore, $S(x,t)$ satisfies
 \bess\left\{\begin{array}{lll}
 S_t-d\Delta S\ge [a-b-c(S+I)-\ell P-kI]S,\;&x\in\oo,\;0< t<T,\\[1mm]
 \displaystyle\frac{\partial S}{\partial\nu}=0, \ &x\in\partial\Omega, \ 0<t<T,\\[2mm]
 S(x,0)=S_0(x)>0,\ &x\in\Omega.
 \end{array}\right.
 \eess
It follows that $S(x,t)>0$ for $x\in\bar\oo$ and $0<t<T$ by the strong maximum principle. Similarly, we have that $P(x,t)>0$ for $x\in\bar\oo$ and $0<t<T$.

{\it Step 2}. Adding the first two differential equations of \qq{1.1} we see that $S+I$ satisfies
 \bess\left\{\begin{array}{lll}
 (S+I)_t-d\Delta (S+I)\le [a-b-c(S+I)](S+I),\;&x\in\oo,\;0<t<T,\\[1mm]
 \dd\frac{\partial(S+I)}{\partial\nu}=0, \ &x\in\partial\Omega, \ 0<t<T,\\[2mm]
 S(x,0)+I(0,x)=S_0(x)+I_0(x)>0,\ &x\in\Omega.
 \end{array}\right.
 \eess
It is deduced that
 \[S(x,t)+I(x,t)\le\max\kk\{\max_{\bar\oo}\{S_0(x)+I_0(x)\},\,(a-b)/c\rr\}
 \;\;\;\text{for}\;\;x\in\oo,\;0\le t<T\]
by the maximum principle. Let $\delta=\min\big\{\ell/\theta,\,\gamma/\sigma\big\}$ and
 \[W(x,t)=S(x,t)+I(x,t)+\delta P(x,t).\]
Then we have that, by the direct calculations,
 \bess
 W_t-\Delta(dS+dI+D\delta P)+\rho W&\le& [a-b+\rho-c(S+I)](S+I)\\[1mm]
 &\le& \frac{(a+\rho-b)^2}{4c}\eess
for $x\in\oo$ and $0\le t<T$. The integrating over $\oo$ yields
 \[\frac{{\rm d}}{{\rm d}t}\int_\oo W\dx+\rho\int_\oo W\dx\le \frac{(a+\rho-b)^2}{4c}|\oo|.\]
Thus,
 \bess
 \int_\oo W\dx\le \int_\oo W_0(x)\dx+\dd\frac{(a+\rho-b)^2}{4c\rho}\;\;\;\text{for}\;\;0\le t<T.
 \eess
Because we already know that $S, I, P>0$ in $\oo\times[0,T)$, the above estimate implies
 \[\delta\int_\oo P\dx\le \int_\oo W_0(x)\dx+\dd\frac{(a+\rho-b)^2}{4c\rho}\;\;\;\text{for}\;\;0\le t<T.\]
By use of \cite[Proposition 1]{Rothe} or \cite[Theorem 2.14]{Wpara} we have that there exists a constant $C$ depends only on the parameters appearing in the differential equations of \qq{1.1} and the maximum norms of the initial data $S_0(x)$, $I_0(x)$ and $P_0(x)$ such that
 $$\sup_{0\leq t<T}\|P(\cdot,t)\|_{L^\infty(\Omega)}\leq C.$$
Consequently, $T=\infty$ and \qq{2.1} holds for all $t>0$ by the standard theory of parabolic partial differential equations.

Utilizing \cite[Theorem 2.13]{Wpara} and \cite[Theorem 2.11]{Wpara} in turn, we can show that $S,\,I,\,P\in{C^{2+\alpha,\,1+\alpha/2}(\overline\Omega\times[\tau, \infty))}$ and the estimate \qq{2.2} holds. \end{proof}

\begin{theo}\lbl{t3.1} The problem \eqref{1.2} has a unique global positive solution $(S, I, P)$ and there exists a positive constant $C$ depends only on the parameters appearing in the differential equations of \qq{1.2} and the maximum norms of the initial data $S_0(x)$, $I_0(x)$ and $P_0(x)$ such that
 \bes
 0< S,\; I\le\max\kk\{\dd\max_{\bar\oo}\{S_0(x)+I_0(x)\},\,\dd\frac{a-b}c\rr\},
  \;\;\; 0< P\le C\;\;\;\text{for}\;\;x\in\oo,\; t>0.
  \lbl{2.2a}\ees
Moreover, for any given $0<\alpha<1$ and $\tau>0$, $S,\,I,\,P\in C^{2+\alpha,\,1+\alpha/2}(\overline\Omega\times[\tau, \infty))$, and there exists a constant $C(\tau)$ such that
\bes
\|S,\,I,\,P\|_{C^{2+\alpha,\,1+\alpha/2}(\overline\Omega\times[\tau, \infty))}\leq C(\tau).
 \lbl{2.3}\ees
\end{theo}

The proof of this theorem is similar to that of Theorem \ref{t2.1} and we omit the details here.

\begin{remark} If $d=D$, then the constant $C$ in \qq{2.1} and \qq{2.2a} can be chosen as
 \[C=\frac 1\delta
 \max\left\{\max_{\bar\oo}\big[S_0(x)+I_0(x)+\delta  P_0(x)\big],\; \frac{(a+\rho-b)^2}{4c\rho}\right\},\]
where $\delta=\min\big\{\ell/\theta,\,\gamma/\sigma\big\}$.
 \end{remark}

\section{Stabilities of nonnegative constant equilibrium solutions of \qq{1.1} for the special case $(\gamma,\sigma)=(\ell, \theta)$}\lbl{s2.2}

This section deals with the stabilities of nonnegative constant equilibrium solutions of \qq{1.1} for the special case $(\gamma,\sigma)=(\ell, \theta)$. In this a situation, the problem \qq{1.1} becomes
\bes
 \left\{\begin{array}{lll}
 S_t-d\Delta S=a(S+I)-b S-c(S+I)S-kIS-\ell SP,\;\; &x\in\Omega, \
 t>0,\\[1mm]
  I_t-d\Delta I=k IS-b I-c(S+I)I-\ell IP,\ &x\in\Omega, \ t>0,\\[1mm]
 P_t-D\Delta P=\theta (S+I)P-\rho P,\ &x\in\Omega, \ t>0,\\[1mm]
 \dd\frac{\partial S}{\partial\nu}=\frac{\partial I}{\partial\nu}=
 \frac{\partial P}{\partial\nu}=0, \ &x\in\partial\Omega, \ t>0,\\[2mm]
 S(x,0)=S_0(x)>0,\ I(x,0)=I_0(x)>0,\ P(x,0)=P_0(x)>0,\ &x\in\bar\Omega.
 \lbl{3.1}\end{array}\right. \ees
The possible nonnegative constant equilibrium solutions (CESs) of \qq{3.1} are
 \bess
 &&\text{two trivial CESs}:\, E_0=(0,\,0,\,0), \;\;\;E_1=\left(\dd\frac{a-b}c,\,0,\,0\right),\\
&&\text{disease-free CES}:\,E_I=\left(\dd\frac\rho\theta,\,0,\,\dd\frac {a-b-c\rho/\theta}\ell\right)\;\;\text{when}\; a>b+\dd\frac{c\rho}\theta, \\
&&\text{predator-free CES}:\,E_P=\left(\dd\frac ak,\,\dd\frac{a(k-c)-bk}{kc},\,0\right)\;\;\text{when}\; a>\dd\frac{bk}{k-c},\;k>c,\\
 &&\text{endemic (positive) CES}:\,E_*=\left(\dd\frac ak,\, \dd\frac{k\rho-a\theta}{k\theta},\,\dd\frac{a-b-c\rho/\theta}\ell\right)\;\;\text{when}\;
b+\dd\frac{c\rho}\theta<a<\dd\frac{k\rho}\theta.
 \eess

In this section we shall give a complete conclusion about the stabilities of nonnegative constant equilibrium solutions $E_0$, $E_1$, $E_I$, $E_P$ and $E_*$. The results show that such a problem has neither periodic solution nor Turing patterns.

\begin{theo}\lbl{t5.1} Assume that $\partial\oo\in C^{2+\alpha}$ for some $0<\alpha<1$.  Regarding the stabilities of nonnegative constant equilibrium solutions with respect to \qq{3.1}, we have the following conclusions:

{\rm(i)} The trivial constant equilibrium solution $E_0$ is unstable.

{\rm(ii)} If either $a>b+c\rho/\theta$, or $a>bk/(k-c)$ and $k>c$, then the trivial constant equilibrium solution $E_1$ is unstable. If $a\le b+c\rho/\theta$ and $a\le bk/(k-c)$, then $E_1$ is globally asymptotically stable.

{\rm(iii)} Assume that the disease-free constant equilibrium solution $E_I$ exists, i.e., $a>b+c\rho/\theta$. If $a<\frac{\rho}\theta k$, then $E_I$ is unstable. If $a\ge\frac{\rho}\theta k$, then $E_I$ is globally asymptotically stable.

{\rm(iv)} Assume that the predator-free constant equilibrium solution $E_P$ exists, i.e., $a>bk/(k-c)$ and $k>c$. If $a>b+c\rho/\theta$, then $E_P$ is unstable. If $a\le b+c\rho/\theta$, then $E_P$ is globally asymptotically stable

{\rm(v)} The endemic constant equilibrium solution $E_*$ is globally asymptotically stable when it exists.
\end{theo}

\subsection{Global asymptotic behaviors of the diffusive prey-predator model}

In the process of proving Theorem \ref{t5.1}, the global asymptotic behaviors of the diffusive classical prey-predator model will be involved. The diffusive classical prey-predator model with constant coefficients and the homogeneous Neumann boundary conditions reads as
 \bes\left\{\begin{array}{ll}
 u_t-d\Delta u=b(a-u)u-c uv,\;\;&x\in\oo,\;t>0,\\[1mm]
 v_t-D\Delta v=k(u-h)v,\;\;&x\in\oo,\;t>0,\\ [1mm]
 \dd\frac{\partial u}{\partial\nu}=\dd\frac{\partial v}{\partial\nu}=0, \ &x\in\partial\Omega, \ t>0,\\[1.5mm]
 u(x,0)=u_0(x)>0,\ v(x,0)=v_0(x)> 0,\ &x\in\bar\Omega,
 \end{array}\right.
 \label{3.2}\ees
where parameters are positive constants. The following theorem gives the global asymptotic stabilities of $(a, 0)$ and  $\big(h, \frac{b(a-h)}c\big)$.

\begin{theo}\lbl{t3.2} Suppose that $\partial\oo\in C^2$, and the initial data $u_0, v_0\in W^2_p(\oo)$ for some $p>1$ and satisfy the compatibility conditions:
 \[\dd\frac{\partial u_0}{\partial\nu}=\dd\frac{\partial v_0}{\partial\nu}=0\;\;\;\text{on}\;\;\partial\Omega.\]
Let $(u, v)$ be the unique solution  of \qq{3.2}. Then we have
 \begin{enumerate}[label={\rm (\roman*)}]
	\item $\dd\lim_{t\to\infty}(u, v)=(a, 0)\;\;\;\text{in}\;\;  C^2(\bar\Omega)\;\;\;\text{when}\;\;h\ge a,$
	\item $\dd\lim_{t\to\infty}(u, v)=\dd\left(h, \frac{b(a-h)}c\right)\;\;\;
	\text{in}\;\; C^2(\bar\Omega)\;\;\;\text{when}\;\;h<a.$
\end{enumerate}
\end{theo}

\begin{proof} By the standard theory for parabolic equations we can show that the problem \qq{3.2} has a unique global solution $(u,v)$, and $u>0$, $v>0$. Moreover, similar to the proof of Theorem \ref{t2.1}, we can show that for any given $\tau>0$, there is a constant $C(\tau)$ such that
 \bes
\|u,\,v\|_{C^{2+\alpha,\,1+\alpha/2}(\overline\Omega\times[\tau, \infty))}\leq C(\tau).
 \lbl{3.3}\ees

Now we prove the conclusion (i). Denote
 \bess
 V(t)=\int_\Omega\left(u-a-a\ln\frac{u}a+\frac ck v\right){\rm d}x.
 \eess
The direct calculation yields
  $$V'(t)=-b\int_\Omega (u-a)^2{\rm d}x-\int_\Omega \dfrac{ad|\nabla u|^2}{u^2}{\rm d}x-c\int_\Omega (h-a)v{\rm d}x\le -b\int_\Omega (u-a)^2{\rm d}x\le 0,$$
which implies that the limit $\lim_{t\to\infty}V(t)=V^*$ exists. The estimate \qq{3.3} shows that the function $\psi(t)=\int_\Omega (u-a)^2{\rm d}x$ is continuously differentiable and $\psi'(t)\le C$ in $[\tau,\infty)$ for some constant $C>0$. It is deduced by \cite[Lemma 1.1]{Waml18} that $\lim\limits_{t\to \infty}\int_\Omega (u-a)^2{\rm d}x=0$. This combining with the estimate \qq{3.3} allows us to derive
\bes
 \lim_{t\to\infty}u=a\;\;\;\text{in}\;\;C^2(\bar \Omega),
 \lbl{3.4}\ees
and
  \bes
  V^*=\lim_{t\to\infty}V(t)=\lim_{t\to\yy}\frac ck\int_\Omega v(x,t){\rm d}x.
  \lbl{3.5}\ees
Thus, there exists $t_i\to\infty$ such that $\dd\lim_{i\to\infty}\frac{\rm d}{{\rm d} t}\int_\Omega u(x,t_i){\rm d}x=0$. By the first equation of \eqref{3.2},
 $$	b\int_\Omega(a-u)u|_{t=t_i}{\rm d}x-c\int_\Omega uv|_{t=t_i}{\rm d}x
 =\frac{\rm d}{{\rm d} t}\int_\Omega u(x,t_i){\rm d}x\to 0\quad \text{as}\;\; i\to\infty.$$
Therefore, $\lim\limits_{i\to\infty}\int_\Omega v(x,t_i){\rm d}x=0$.  Thus $V^*=0$ by \qq{3.5}. Then we have, by the estimate \eqref{3.3},
\bess
  \lim_{t\to\infty}v(x,t)=0\;\;\;\text{in}\;\;C^2(\bar\Omega).
\eess

In the following we prove the conclusion (ii).

{\it Step 1}. Denote $\frac{b(a-h)}c=\tilde v$ and define
 \[F(t)=\int_\oo\left(u-h-h\ln\dd\frac{u}{h}\right)\dx+\frac c k\int_\oo\left(v-\tilde v-\tilde v\ln\dd\frac{v}{\tilde v}\right)\dx.\]
Then the direct calculation yields
\bess
 F'(t)&=&-b\int_\oo(u-h)^2\dx-\int_\oo\left(\frac{dh|\nabla u|^2}{u^2}+\frac{c D\tilde{v}|\nabla v|^2}{ k v^2}\right)\dx\\[1mm]
 &\le& -b\int_\oo(u-h)^2\dx-\frac{c D\tilde{v}}{ k C^2}\int_\oo|\nabla v|^2\dx\\[1mm]
 &\le& 0.
\eess
Consequently, the limit $\lim_{t\to\infty}F(t)=F^*$ exists. Moreover, the estimate \qq{3.3} shows that functions $\psi_1(t)=\int_\Omega (u-h)^2{\rm d}x$ and $\psi_2(t)=\int_\Omega|\nabla v|^2{\rm d}x$ are continuously differentiable and $\psi'_i(t)\le C$ in $[\tau,\infty)$ for some constant $C>0$. It then follows by \cite[Lemma 1.1]{Waml18} that
 \bess
 \lim_{t\to\infty}\int_\oo(u-h)^2\dx=0,\;\;\; \text{and}\;\; \lim_{t\to\infty}\int_\Omega|\nabla v|^2\dx=0,
 \eess
and using \qq{3.3} again we have
\bes
 \lim_{t\to\infty}\|u-h\|_{C^2(\bar\oo)}=0,\;\;\;\lim_{t\to\infty}\|\nabla v\|_{C^1(\bar\oo)}=0.
 \lbl{3.6}\ees
This implies that
 \bes
 F^*=\lim_{t\to\infty}F(t)=\frac c k\lim_{t\to\infty}\int_\oo\left(v-\tilde v-\tilde v\ln\dd\frac{v}{\tilde v}\right)\dx,
 \lbl{3.7}\ees
and there exist $t_i\to\infty$ such that
 \[\lim_{i\to\infty}\frac{{\rm d}}{{\rm d}t}\int_\oo u(x,t_i)\dx=0.\]
It follows from the first equation of \eqref{3.2} that, at $t=t_i$,
 \bess
 0=\lim_{i\to\infty}\frac{{\rm d}}{{\rm d}t}\int_\oo u\dx
 =\lim_{i\to\infty}\int_\oo(u-h)[ab-b(u+h)-cv]\dx-c h\lim_{i\to\infty}\int_\oo(v-\tilde v)\dx,\eess
which implies $\lim\limits_{i\to\infty}\int_\oo(v(x,t_i)-\tilde v)\dx=0$. This combining with \qq{3.3} and the second limit of \eqref{3.6} shows that
 \bess
 \lim_{i\to\infty}\|v(\cdot, t_i)-\tilde v\|_{C^2(\bar\oo)}=0.
   \eess
Then, using \qq{3.7}, we have $F^*=0$, i.e.,
 \bes
 \lim_{t\to\infty}\int_\oo\left(v-\tilde v-\tilde v\ln\dd\frac{v}{\tilde v}\right)\dx=0.
 \lbl{3.8}\ees

{\it Step 2}. We shall prove
 \bes
 \lim_{t\to\infty}\int_\Omega|v-\tilde v|{\rm d}x=0.
  \lbl{3.9}\ees
Assume on contrary that there exist $s_j\to \infty$ such that
\bess
 \lim_{j\to\infty}\int_\Omega|v(x,s_j)-\tilde v|{\rm d}x>0.
  \eess
Then there exist $x_j\in\bar\oo$, $0<\delta<\tilde v$ and $j_0\gg1 $ such that
 \bes
 v(x_j,s_j)-\tilde v=\max_{\bar\Omega}(v(x,s_j)-\tilde v)\ge 2\delta\;\;\;\text{for}\;\; j\ge j_0,\lbl{3.10}\ees
or
  \bes
  v(x_j,s_j)-\tilde v=\min_{\bar\Omega}(v(x,s_j)-\tilde v)\le -2\delta\;\;\;\text{for}\;\; j\ge j_0.\lbl{3.11}\ees
Then using the second limit of \qq{3.6} we have that, for sufficiently large $T$,
 \bess
 \min_{\bar\Omega}v(\cdot,s_j)\ge v(x_j,s_j)-\max_{\bar\oo}|\nabla v(\cdot, s_j)|{\rm diam}(\oo)\ge\tilde v+\delta, \;\; \forall\ s_j\ge T
 \;\;\; \text{when \qq{3.10} holds},\\
 \max_{\bar\Omega}v(\cdot,s_j)\le v(x_j,s_j)+\max_{\bar\oo}|\nabla v(\cdot, s_j)|{\rm diam}(\oo)\le\tilde v-\delta, \;\; \forall\ s_j\ge T \;\;\;
 \text{when \qq{3.11} holds}.
 \eess
By note that the formula $v-\tilde v-\tilde v\ln\dd\frac{v}{\tilde v}$ is increasing in $v$ when $v>\tilde{v}$, and is decreasing in $v$ when $v<\tilde{v}$, we have that
	\[v(x,s_j)-\tilde v-\tilde v\ln\dd\frac{v(x,s_j)}{\tilde v}\ge
\tilde v+\delta-\tilde v-\tilde v\ln\dd\frac{\tilde v+\delta}{\tilde v}=\delta-\tilde v\ln\dd\frac{\tilde v+\delta}{\tilde v}>0\]
	for all $x\in\oo$ when \qq{3.10} holds, and
	\[v(x,s_j)-\tilde v-\tilde v\ln\dd\frac{v(x,s_j)}{\tilde v}\ge
	\tilde v-\delta-\tilde v-\tilde v\ln\dd\frac{\tilde v-\delta}{\tilde v}=-\delta-\tilde v\ln\dd\frac{\tilde v-\delta}{\tilde v}>0\]
	for all $x\in\oo$ when \qq{3.11} holds. This contradicts with the limit \qq{3.8}. So the limit \qq{3.9} holds. By use of the estimate \qq{3.5} again, it follows that $\lim_{t\to\infty}v=\tilde v$ in $C^2(\bar\Omega)$.
\end{proof}

\subsection{The proof of Theorem \ref{t5.1}}

We are ready to prove the Theorem \ref{t5.1}. Let $0=\mu_1<\mu_2<\mu_3<\cdots$ be eigenvalues of the
operator $-\Delta$ in $\Omega$ with the homogeneous Neumann boundary
condition, and $E(\mu_i)$ be the eigenspace corresponding to
$\mu_i$ in $H^1(\Omega)$. Let $\bX=[H^1( \Omega)]^3$,
$\{\phi_{ij}\,;\, j=1,\,\ldots,\,\hbox{\rm dim}\,E(\mu_i)\}$ be an
orthonormal basis of $E(\mu_i)$, and $\bX_{ij}=\{\bc\phi_{ij}:\,
\bc\in\bR^3\}.$ Then,
\bess
\bX=\bigoplus_{i=1}^\infty \bX_i\quad\hbox{and}\quad\bX_i=
\bigoplus_{j=1}^{{\rm dim\,}E(\mu_i)}\bX_{ij}.
\eess

The variational matrix of the reaction terms in \qq{3.1} at the equilibrium solution $(S, I, P)$ is given by
\bess
 M(S, I, P)=\left(\begin{array}{cccc}
 a-b-2cS-cI-\ell P-kI \ & \ a-cS-kS \ & \ -\ell S\\[1mm]
 kI-cI\ & \ kS-b-cS-2cI-\ell P \ & \ -\ell I\\[1mm]
 \theta P \ & \ \theta P \ &\ \theta(S+I)-\rho
 \end{array}\right).
 \eess
Let
  \[{\cal D}={\rm diag}(d,\,d,\,D)\;\;\;\text{and}\;\;{\cal L}={\cal
  D}\Delta+M(S, I, P).\]
For each $i\geq 1$, $\bX_i$ is invariant under the operator ${\cal L}$, and $\lambda$ is an eigenvalue of ${\cal L}$  if and only if
it is an eigenvalue of the matrix $-\mu_i{\cal D}+M(S, I, P)$ for some $i\geq 1$.

(i)\, It is easy to see by the linearization method that the trivial constant equilibrium solution $E_0$ is unstable (\cite[Theorem 2.2\,(i)]{HWang22}).

(ii)\, Using the linearization method, it is not difficult to deduce that if either $a>b+c\rho/\theta$, or $a>bk/(k-c)$ and $k>c$, then the trivial constant equilibrium solution $E_1$ is unstable (\cite[Theorem 2.2\,(ii)]{HWang22}).

Now we prove that $E_1$ is globally asymptotically stable if $a\le b+c\rho/\theta$ and $a \le bk/(k-c)$. Letting $S+I=u$, and adding the first two equations of \eqref{3.1} we have
 \bes\left\{\begin{array}{ll}
 u_t=d\Delta u+(a-b-cu)u-\ell uP,\;\;&x\in\oo,\;t>0,\\[1mm]
 P_t=D\Delta P+(\theta u-\rho) P,\;\;&x\in\oo,\;t>0,\\[1mm]
 \dd\frac{\partial u}{\partial\nu}=\dd\frac{\partial P}{\partial\nu}=0, \ &x\in\partial\Omega, \ t>0,\\[2mm]
 u(x,0)=S_0(x)+I(x,0)> 0,\ P(x,0)> 0,\ &x\in\Omega.
\end{array}\right.\label{3.12}
 \ees
Since $a\le b+c\rho/\theta$, and by using Theorem \ref{t3.2}, we have
 $$(u,P)\to((a-b)/c,0)\ \text{in}\ C^2(\bar\Omega)\quad \text{as}\ t\to\infty.$$
By the second equation of \eqref{3.1} we have
  \[I_t-d\Delta I=I\big[(k-c)u-b-\ell P-kI\big]\le I\big [(k-c)u-b-kI].\]
Noticing that $\lim_{t\to\infty}u=(a-b)/c$ in $C(\boo)$ and the condition $a\le bk/(k-c)$, for any $\varepsilon>0$ and sufficiently large $T$, we have
	$$I_t-d\Delta I \le I\left[\frac{(a-b)(k-c)}{c}+\varepsilon-b-kI\right]
   \le I(\varepsilon-k I),\quad x\in\Omega,\; t>T.$$
Hence, it follows from the maximum principle that $\limsup_{t\to\infty}\max_{\bar\oo}I\le \varepsilon/k$. As $\varepsilon$ was arbitrarily chosen, we have $\lim_{t\to\infty}\max_{\bar\oo}I= 0$. Using the estimate \qq{2.2} we still  have that $\lim_{t\to\infty}u=(a-b)/c$ and $\lim_{t\to\infty}I=0$ in $C^2(\bar\oo)$.

\sk(iii)\, Using the linearization method we can see that if $a<\rho k/\theta$ then the disease-free constant equilibrium solution $E_I$ is unstable (\cite[Theorem 2.2\,(iii)]{HWang22}).

Next, we prove that if $a\ge\rho k\theta$, $E_I$ is globally asymptotically stable. By note that $E_I$ exists only if $a>b+c\rho/\theta$, apply Theorem \ref{t3.2} to the problem \qq{3.12} to deduce that
\begin{equation}\label{3.16}
	\lim_{t\to\infty}(u,P)=\left(\frac\rho \theta, \frac{(a-b-c\rho/\theta)}\ell\right)
\;\;\; \text{in}\;\; C^2(\bar\Omega).
\end{equation}
Now, consider the second equation of \eqref{3.1}:
\[I_t-d\Delta I=\big[(k-c)u-b-\ell P-kI\big]I.\]
Making use of the limits \eqref{3.16} we have that, in $C^2(\bar\oo)$,
  $$\lim_{t\to\infty}[(k-c)u-b-\ell P]={\rho k}/\theta-a\le 0$$
since $a\ge\rho k/\theta$. Therefore,
$$\lim_{t\to\infty}I=0\;\;\;\text{in}\;\;C^2(\bar\oo),$$
and hence $\lim_{t\to\infty} S=\rho/\theta$ in $C^2(\Omega)$.

\sk (iv)\, It is easy to know by the linearization method that if $a>b+c\rho/\theta$ then the predator-free constant equilibrium solution $E_P$ is unstable (\cite[Theorem 2.2\,(iv)]{HWang22}).

It will be proved that $E_P$ is globally asymptotically stable if $a\le b+c\rho/\theta$.
Same as Case (ii), $\lim_{t\to\infty}(S+I)=(a-b)/c$ and $\lim_{t\to\infty}P=0$ in $C^2(\bar\oo)$. Set $u=S+I$. Then $I$ satisfies
 \bes\left\{\begin{array}{ll}
 I_t-d\Delta I=I\big[(k-c)u-b-\ell P-kI\big],\;\;&x\in\oo,\; t>0,\\[1mm]
 \dd\dd\frac{\partial I}{\partial\nu}=0, \ &x\in\partial\Omega, \ t>0,\\[2mm]
 I(x,0)>0,\ &x\in\Omega,
 \end{array}\right.
 \lbl{3.17}\ees
and
 $$\lim_{t\to\infty}[(k-c)u-b-\ell P]=\frac{k(a-b)-ac}{c}\;\;\;\text{in}\;\;C(\bar\oo).$$
Make use of the comparison arguments, it can be shown that
 $$\lim_{t\to\infty}I=\frac{k(a-b)-ac}{kc},\;\;\;\text{and}\;\;
 \lim_{t\to\infty}S=\frac ak\;\;\;\text{in}\;\; C(\boo).$$
In view of the estimate \qq{2.2} we can also see that these two limits hold in $C^2(\bar\oo)$.

\sk (v)\, We prove that the endemic constant equilibrium solution $E_*$ is globally asymptotically stable when it exists. Set $u=S+I$. Then \qq{3.12} and \eqref{3.16} hold, and $I$ satisfies \qq{3.17}. Thus, by use of \eqref{3.16} we have that
  $$\lim_{t\to\infty}[(k-c)u-b-\ell P]=\frac{\rho k-a\theta}\theta\;\;\;\text{in}\;\;C(\bar\oo).$$
Similar to Case (iv), it can be derived that
 $$\lim_{t\to\infty}I=\frac{k\rho-a\theta}{k\theta}, \;\;\;\text{and}\;\; \lim_{t\to\infty}S=\frac ak\;\;\;\text{in}\;\;C^2(\bar\oo).$$
The proof is complete.

\section{Nonnegative solutions of \qq{1.3} and their global asymptotic stabilities}\lbl{s2}
\setcounter{equation}{0} {\setlength\arraycolsep{2pt}

In this section, we first establish the necessary and sufficient conditions for the existence of positive solutions of problems \qq{1.5}, \qq{1.6} and \qq{1.3} respectively, and prove that the positive solutions are unique when they exist. Especially, for
problems \qq{1.5} and \qq{1.6}, it is also proved that the unique positive solution is non degenerate. Then we study the global asymptotic stabilities of trivial and semi-trivial nonnegative equilibrium solutions $(0,0,0)$, $(S^*,0,0)$ and $(\tilde S, \tilde I,0)$.

We state a well known conclusion here, which will be used frequently later: \vspace{-2mm}
 \begin{enumerate}[leftmargin=3em]
 \item[$\bullet$]\; The strong maximum principle holds for the operator $-d\Delta +q$ in $\oo$ with the homogeneous Dirichlet boundary conditions if and only if $\lm_1^d(q)>0$; the principal eigenvalue $\lm_1^d(q)$ is strict increasing in $d$ and $q$.
  \end{enumerate}

\subsection{Abstract results}

To study the existence of positive solutions of \qq{1.3}, \qq{1.5} and \qq{1.6}, we state some abstract results about topological degree in cones and give estimates of positive solutions of \qq{1.3} in this part. Consider the boundary value problem
 \bes
 \left\{\begin{array}{lll}
 \mathscr{L}u=f(u),\;\; &x\in\Omega, \\[1mm]
 u=0, \ &x\in\partial\Omega,
 \lbl{4.1}\end{array}\right.
 \ees
where
 $$\mathscr{L}={\rm diag}(-d_1\Delta,\cdots,-d_m\Delta),\;\; u=(u_1, \cdots u_m),\;\;\;f=(f_1,\cdots,f_i),$$
and $f_i\in C^1$, $f_i(u)\big|_{u_i=0}\ge 0$ for all $u\ge 0$.

For the large positive constant $M$, we define
 \bess
  E&=&X^m \;\;\text{with}\;\; X=\{z\in C^1(\bar\Omega):\, z|_{\partial\Omega}=0\},
  \nonumber\\
  W&=&K^m\;\;\text{with}\;\; K=\{z\in C^1(\bar\Omega):\, z|_{\Omega}\geq 0,\, z|_{\partial\Omega}=0\}, \nonumber\\
  F(u)&=&\big(M+\mathscr{L}\big)^{-1}\big(f(u)+Mu\big).
  \eess
Then $u\in W$ is a solution of \qq{4.1} if and only if $F(u)=u$. For $u\in W$, we define
 \bess
 W_u&=&\{v\in E:\, \exists\; r>0\;\; \text{s.t.}\;\;u+tv\in
W,\,\forall\; 0\leq t\leq r\},\\
 S_u&=&\{v\in\overline{W}_u:-v\in\overline{W}_u\}.
 \eess

\begin{defi}\label{d5.2} Let $T:
\overline{W}_u\to\overline{W}_u$ be a compact linear
operator. $T$ is said to have Property $\alpha$ if there exist
$t\in(0,1)$ and $w\in\overline{W}_u \setminus S_u$ such that
$w-tTw\in S_u$.
\end{defi}

\begin{theo}\label{t5.2}{\rm(\cite{Dancer, Lil, Wang93})}\; Assume that $u\in W$ is an isolated fixed point of $F$ and $I-F'(u)$ is invertible in $\overline{W}_u$. Then the following statements hold.
\begin{enumerate}[leftmargin=3em]
\item[{\rm(1)}]\, If $F'(u)$ has Property $\alpha$, then ${\rm index}_W(F,u)=0$;
\item[{\rm(2)}]\, If $F'(u)$ does not have Property $\alpha$, then
  $${\rm index}_W(F,u)={\rm index}_E(F,u)=(-1)^{\beta},$$
where $\beta$ is the sum of algebraic multiplicities of all eigenvalues of $F'(u)$ that are greater than one.
 \end{enumerate}
\end{theo}

\begin{coro}\label{c4.1} Assume that $u\in W$ is an isolated fixed point of $F$, and the problem
 \bes
 F'(u)\phi=\phi,\;\;\phi\in \overline{W}_u
 \lbl{4.2}
 \ees
has only the zero solution. Then the following statements hold.
\begin{enumerate}[leftmargin=3em]
\item[{\rm(1)}]\, If the eigenvalue problem
   \bes
 F'(u)\phi=\lm\phi,\;\;\phi\in\overline{W}_u\setminus S_u
 \lbl{4.3}
 \ees
has an eigenvalue $\lm>1$, then ${\rm index}_W(F,u)=0$;

\item[{\rm(2)}]\, If all eigenvalues of the eigenvalue problem
   \bes
 F'(u)\phi-\lm\phi\in S_u,\;\;\phi\in\overline{W}_u\setminus S_u
 \lbl{4.4}
 \ees
are less than $1$, then
 $${\rm index}_W(F,u)={\rm index}_E(F,u)=(-1)^{\beta},$$
where $\beta$ is the sum of algebraic multiplicities of all eigenvalues
of $F'(u)$ that are greater than one.
 \end{enumerate}
\end{coro}

\begin{proof} The condition that \qq{4.2} has only the zero solution shows that $I-F'(u)$ is invertible in $\overline{W}_u$.

(1) Let $\lm>1$ be an eigenvalue of \qq{4.3} and $\phi$ be the corresponding eigenfunction. Then $\phi\in\overline{W}_u\setminus S_u$ and $\phi-\frac 1\lm F'(u)\phi=0\in S_u$. As $0<\frac 1\lm<1$, we see that $F'(u)$ has Property $\alpha$, and so ${\rm index}_W(F,u)=0$ by Theorem \ref{t5.2}.

(2) We shall show that $F'(u)$ does not have Property $\alpha$. On the contrary we assume that there exist $0<t<1$ and $w\in\overline{W}_u \setminus S_u$ such that
$w-tF'(u)w\in S_u$. Then $F'(u)w-\frac 1t w\in S_u$ and $\lm=1/t>1$ is an eigenvalue of \qq{4.4}. We get a contradiction and the conclusion is followed by Theorem \ref{t5.2}.
\end{proof}

In order to compute the topological degree of the operator $I-F$, we first consider the following auxiliary problem
 \bes
 \left\{\begin{array}{lll}
 \mathscr{L}u=\tau f(u),\;\; &x\in\Omega, \\[1mm]
 u=0, \ &x\in\partial\Omega
 \end{array}\right.\lbl{4.5}\ees
with $0\leq\tau\leq 1$. Assume that there exists a positive constant $C_1$ such that
any non-negative solution $u_\tau$ of \eqref{4.5} satisfies
 \bess
 \|u_\tau\|_{C(\bar\Omega)}\leq C_1.
 \eess
Applying the $L^p$ theory and embedding theorem to \qq{4.5}, there exists a positive constant $C$ such that any non-negative solution $u_\tau$ of \eqref{4.5} satisfies
 \bes
 \|u_\tau\|_{C^1(\bar\Omega)}\leq C.
  \lbl{4.6}\ees
 Set
 \begin{eqnarray*}
 {\cal O}=\{u\in W:\, \|u\|_{C^1(\bar\Omega)}<C+1\}.
  \end{eqnarray*}
We can enlarge $M$ such that
 \bess
 \tau f_i(u)+Mu_i\geq 0,\;\;\;\forall\, u\in {\cal O},\;0\leq \tau\leq 1.
 \eess
Define
 $$ F_\tau(u)=(\mathscr{L}+M)^{-1}\left(\tau f(u)+Mu\right),\;\;u\in E,\; 0\leq \tau\leq 1.$$
It is easy to deduce from the homotopy invariance of the topological degree that
 \bes
 {\rm deg}_W(I-F,{\mathcal O})={\rm deg}_W(I-F_0,{\mathcal O})
 ={\rm index}_W(F_0,\boldsymbol{0})=1.
 \lbl{4.7}\ees

In order to study the solutions of boundary value problems by using topological degree theory, the a prior estimate \qq{4.6} is crucial. In the end of this section, we shall  prove that the estimate \qq{4.6} is valid for the problem \qq{1.3}. Once this is done, it is easy to see that the estimate \qq{4.6} also holds true for problems \qq{1.5} and \qq{1.6}.

Consider the following auxiliary problem
\bes
 \left\{\begin{array}{lll}
-d\Delta S=\tau f_1(S,I,P),\;\; &x\in\Omega, \\[1mm]
 -d\Delta I=\tau f_2(S,I,P),\ &x\in\Omega,\\[1mm]
 -D\Delta P=\tau f_3(S,I,P),\ &x\in\Omega,\\[1mm]
  S=I=P=0, \ &x\in\partial\Omega,
 \end{array}\right.
 \lbl{4.8}\ees
where
 \bes\left\{\begin{array}{lll}
 f_1(S, I, P)&=&a(S+I)-b S-c(S+I)S-kSI-\ell S P,\\[1mm]
 f_2(S, I, P)&=&k SI-b I-c(S+I)I-\ell IP,\\[1mm]
 f_3(S, I, P)&=&\theta (S+I)P-\rho P.
 \end{array}\right.
 \lbl{4.9}\ees
Let $(S_\tau, I_\tau, P_\tau)$ be a non-negative solution of \qq{4.8}. If $\tau=0$, then $S_\tau=I_\tau=P_\tau=0$. If $0<\tau\le 1$, adding the first two equations of \qq{4.8} we have
 \bess
 \left\{\begin{array}{lll}
-d\Delta(S_\tau+I_\tau)\le \tau[(a-b)(S_\tau+I_\tau)-c(S_\tau+I_\tau)^2],\;\; &x\in\Omega, \\[1mm]
 S_\tau+I_\tau=0, \ &x\in\partial\Omega.
 \end{array}\right.
 \eess
It follows that
 $$\max\limits_{\bar\oo}[S_\tau(x)+I_\tau(x)]\le(a-b)/c.$$
Set $z_\tau=d(S_\tau+I_\tau)+\frac{D\ell}\theta P_\tau$. Then, through the carefully calculations,
  \bess
 \left\{\begin{array}{lll}
-\Delta z_\tau=\tau\big[(a-b)(S_\tau+I_\tau)-c(S_\tau+I_\tau)^2-\frac{\ell\rho}\theta P_\tau\big],\;\; &x\in\Omega, \\[1mm]
  z_\tau=0, \ &x\in\partial\Omega.
 \end{array}\right.
 \eess
Let $z_\tau(x_0)=\dd\max_{\bar\oo}z_\tau(x)>0$. Then $x_0\in\Omega$ and $-\Delta z_\tau(x_0)\ge 0$. So we have
 \bess
 \frac{\ell\rho}\theta P_\tau(x_0)&\le& (a-b)[S_\tau(x_0)+I_\tau(x_0)]-c[S_\tau(x_0)+I_\tau(x_0)]^2\le\frac{(a-b)^2}{4c},\\
 \frac{D\ell}\theta\max_{\bar\oo}P_\tau(x)&\le&\max_{\bar\oo}z_\tau(x)=z_\tau(x_0)\\
 &=&S_\tau(x_0)+I_\tau(x_0)+\frac{D\ell}\theta P_\tau(x_0)\\
 &\le&\frac{a-b}c+\frac{D(a-b)^2}{4c\rho},
 \eess
which implies
  $$\max\limits_{\bar\oo}P_\tau(x)\le\frac{\theta(a-b)}{cD\ell}
  +\frac{\theta(a-b)^2}{4c\ell\rho}.$$
Applying the $L^p$ theory and embedding theorem to \qq{4.8}, we can find a positive constant $C$ such that any non-negative solution of \eqref{4.8} satisfies \qq{4.6} with $u_\tau=(S_\tau,I_\tau, P_\tau)$ for all $0\le\tau\le 1$.

\subsection{The existence, uniqueness and non degeneracy of positive solutions of  \qq{1.5}}\lbl{s3.1.1}

Assume $\lm_1^d(b-a)<0$. Then the non-negative trivial solutions of \eqref{1.5} are $\boldsymbol{0}=(0,0)$ and $(S^*,0)$.
Define $E$ and $W$ as above with $m=2$. It is easy to show that
 \[\overline{W}_{\boldsymbol{0}}=W,\;\;S_{\boldsymbol{0}}=\{\boldsymbol{0}\},
 \;\;\overline{W}_{(S^*,0)}=X\times K,\;\;S_{(S^*,0)}=X\times\{0\}.\]
Take $u=(S, I)$, $d_1=d_2=d$ and
 \bess
 f_1(S,I)=a(S+I)-b S-c(S+I)S-k SI,\;\; f_2(S, I)=k SI-b I-c(S+I)I.
 \eess
Then the estimate \qq{4.6} is valid, and \qq{4.7} holds. Moreover, the direct calculation yields
$$F'(S,I)=(M-d\Delta)^{-1}\left(\begin{array}{cc}a-b-2cS-(k+c)I+M \;\;& a-(k+c)S \\[1mm]
 (k-c)I\;\; &(k-c)S-b-2cI+M\end{array}\right).$$

Take $u=\boldsymbol{0}$ in problems \qq{4.2} and \qq{4.3}. Since $\lm_1^d(b-a)<0$, it is easy to see that the problem \qq{4.2} has only the zero solution, and the problem \qq{4.3} has an eigenvalue $\lm=(M+a-b)/(M+\lm^d_0)>1$. It follows that, by Corollary \ref{c4.1},
 \bes
 {\rm index}_W(F,\boldsymbol{0})=0.
 \lbl{4.10}\ees

\begin{lem}\lbl{l3.1} If $\lm_1^d(b-(k-c)S^*)<0$ then ${\rm index}_W(F, (S^*, 0))=0$.
\end{lem}

\begin{proof} Take $u=(S^*, 0)$ in problems \qq{4.2} and \qq{4.3}. Then
$$F'(S^*, 0)=(M-d\Delta)^{-1}\left(\begin{array}{cc}a-b-2cS^*+M \;\;& a-(k+c)S^* \\[1mm]
 0\;\; &(k-c)S^*-b+M\end{array}\right).$$

We first prove that \qq{4.2} has only the zero solution. In fact, let $(S, I)\in \overline{W}_{(S^*,0)}=X\times K$ is a solution of \qq{4.2}. If $I\not=0$, then $\lm_1^d(b-(k-c)S^*)=0$. This is a contradiction and so $I=0$. Then $S$ satisfies
 \begin{eqnarray*}
\left\{\begin{array}{ll}
-d\Delta S+(b-a+2cS^*)S=0,\;\;&x\in\Omega,\\[1mm]
S=0,\;\;&x\in\partial\Omega.
 \end{array}\right.
 \end{eqnarray*}
Noticing that $S^*$ is the positive solution of \qq{1.4}, we have $\lm_1^d(b-a+cS^*)=0$, and then $\lm_1^d(b-a+2cS^*)>0$. Consequently, $S=0$ by the strong maximum principle.

Now we study the eigenvalue problem
  \begin{eqnarray}
 {\cal A}z:=(M-d\Delta)^{-1}(M-[b-(k-c)S^*])z=\lambda z,\;\; z\in K.
  \label{4.11}\end{eqnarray}
Since $\lm_1^d(b-(k-c)S^*)<0$, we see that $r\left({\cal A}\right)>1$ is the principal  eigenvalue of \qq{4.11}. Let $\phi_2>0$ be the eigenfunction corresponding to $r\left({\cal A}\right)$. Noticing that $r:=r\left({\cal A}\right)>1$ and $\lm_1^d(b-a+2cS^*)>0$, we have that, by the monotonicity of $\lambda_1^d(q)$ in $d$ and $q$,
 $$\lambda_1^{r d}((\lm-1)M+b-a+2cS^*)>\lambda_1^d(b-a+2cS^*)>0.$$
Consequently, the problem
  \bess\left\{\begin{array}{ll}
  r(M-d\Delta)\phi_1=(a-b-2cS^*+M)\phi_1+(a-(k+c)S^*)\phi_2,\;\;&x\in\Omega,\\
  \phi_1=0,&x\in\partial\Omega
  \end{array}\right.\eess
has a unique solution. That is, the problem
  \bess
  (M-d\Delta)^{-1}\big[(a-b-2cS^*+M)\phi_1+(a-(k+c)S^*)\phi_2\big]=r\phi_1,\;\;\phi_1\in X\eess
has a unique solution $\phi_1$. Therefore, $r=r\left({\cal A}\right)>1$ is an eigenvalue of \qq{4.3} and $\phi=(\phi_1, \phi_2)\in\overline{W}_{(S^*,0)}\setminus S_{(S^*,0)}$ is the corresponding eigenfunction. So ${\rm index}_W(F, (S^*, 0))=0$ by Corollary \ref{c4.1}.
\end{proof}

\begin{theo}\lbl{th4.2} The problem \qq{1.5} has a positive solution, denoted by  $(\tilde S, \tilde I)$, if and only if $\lm_1^d(b-a)<0$ and $\lm_1^d(b-(k-c)S^*)<0$. Moreover, the positive solution $(\tilde S, \tilde I)$ is unique and non-degenerate when it exists, and satisfies $\tilde S+\tilde I=S^*$.
\end{theo}

\begin{proof} {\it Step 1: The necessity}. Let $(\tilde S, \tilde I)$ be a positive solution of \qq{1.5}. Then $k>c$ must be true. Let $z=\tilde S+\tilde I$. Then $z$ satisfies \qq{1.4}. So $\lm_1^d(b-a)<0$, and then $z=S^*$ by the uniqueness of positive solutions of \qq{1.4}. It follows from the second equation of \qq{1.5} that $\tilde I$ satisfies
 \bess
 \left\{\begin{array}{lll}
 -d\Delta\tilde I=(k-c)\tilde S\tilde I-b\tilde I-c\tilde I^2<[(k-c)S^*-b]\tilde I,\ &x\in\Omega,\\[1mm]
   S=I=0, \ &x\in\partial\Omega.
  \end{array}\right.
 \eess
As $\tilde I>0$, it follows that $\lm_1^d(b-(k-c)S^*)<0$.

{\it Step 2: The sufficiency}. Let $\lm_1^d(b-(k-c)S^*)<0$.
On the contrary we assume that the problem \qq{1.5} has no positive solution. Then
\[1={\rm deg}_W(I-F,{\mathcal O})={\rm index}_W(F,\boldsymbol{0})+{\rm index}_W(F, (S^*, 0))=0\]
by \qq{4.7}, \qq {4.10} and Lemma \ref{l3.1}. This is a contradiction.

{\it Step 3: The uniqueness}. Assume that $(\tilde S_1, \tilde I_1)$ and $(\tilde S_2, \tilde I_2)$ are two positive solutions of \qq{1.5}. Let $z_1=\tilde S_1+\tilde I_1$ and $z_2=\tilde S_2+\tilde I_2$. Then $z_i$ satisfies \qq{1.4}. Thus, $z_1=z_2=S^*$ by the uniqueness of positive solutions of \qq{1.4}, and then $\tilde I_i$ satisfies
 \begin{eqnarray*}
\left\{\begin{array}{ll}
  -d\Delta\tilde I_i=[(k-c)S^*-b]\tilde I_i-k\tilde I_i^2,\ &x\in\Omega,\\[1mm]
  \tilde I_i=0,\;\;&x\in\partial\Omega.
 \end{array}\right.
 \end{eqnarray*}
It is well known that this problem has at most one positive solution. Consequently, $\tilde I_1=\tilde I_2$, which implies $\tilde S_1=\tilde S_2$.

{\it Step 4: The non-degeneracy}. To prove that the unique positive solution $(\tilde S, \tilde I)$ is non-degenerate, it suffices to show that the problem
\bes
 \left\{\begin{array}{lll}
 -d\Delta S+(b-a+2c\tilde S+(k+c)\tilde I)S=(a-(k+c)\tilde S)I,\;\; &x\in\Omega, \\[1mm]
 -d\Delta I+(b-(k-c)\tilde S+2c\tilde I)I=(k-c)\tilde IS,\;\; &x\in\Omega, \\[1mm]
 S,I\in X,
 \end{array}\right.
 \lbl{4.12}\ees
has only the zero solution.

It has been shown in Theorem \ref{th4.2} that $\tilde S+\tilde I=S^*$ and $\tilde I$ satisfies
  \bess
\left\{\begin{array}{ll}
  -d\Delta\tilde I+[b+k\tilde I-(k-c)S^*]\tilde I=0,\ &x\in\Omega,\\[1mm]
  \tilde I=0,\;\;&x\in\partial\Omega.
 \end{array}\right.
 \eess
Hence $\lm_1^d(b+k\tilde I-(k-c)S^*)=0$, and then
 \bes
 \lm_1^d(b+(k+c)\tilde I-(k-c)\tilde S)>\lm_1^d(b+k\tilde I-(k-c)S^*)=0.
 \lbl{4.13}\ees
Let $u=S+I$. After some straightforward manipulations it can be easily seen that $u$ satisfies
 \bess
 \left\{\begin{array}{lll}
-d\Delta u+(b-a+2cS^*)u=0,\;\; &x\in\Omega, \\[1mm]
  u=0, \ &x\in\partial\Omega.
 \end{array}\right.
 \eess
As $\lm_1^d(b-a+2cS^*)>0$, it is deduced by the strong maximum principle that $u=0$, i.e. $S=-I$.
So we have
\bess
 \left\{\begin{array}{lll}
  -d\Delta I+[b+(k+c)\tilde I-(k-c)\tilde S]I=0,\;\; &x\in\Omega, \\[1mm]
 I=0,\;\; &x\in\partial\Omega.
 \end{array}\right.
 \eess
Noticing that \qq{4.13}. It is deduced by the strong maximum principle that $I=0$, and then $S=0$.
\end{proof}

\subsection{The existence, uniqueness and non degeneracy of positive solutions of  \qq{1.6}}

In this part we study the existence, uniqueness and non degeneracy of positive solutions of \qq{1.6}.

\begin{theo}\lbl{th3.4} The problem \qq{1.6} has a positive solution, denoted by $(\hat S, \hat P)$, if and only if $\lm_1^d(b-a)<0$ and $\lm_1^D\big(\rho-\theta S^*\big)<0$. Moreover, if $N=1$, i.e., the one dimension case, then the positive solution of \qq{1.6} is unique and non-degenerate when it exists.
\end{theo}

\begin{proof} The proof of the sufficiency and necessity can refer to \cite{Li2}. In the following we prove the uniqueness and non-degeneracy of positive solutions. Without loss of generality we assume that $\oo=(0,l)$.

{\it Step 1: The uniqueness}. Let $(S_1,P_1)$ and $(S_2,P_2)$ be two positive solutions of \qq{1.6}, and set $S=S_1-S_2$, $P=P_1-P_2$. Then $(S,P)$ satisfies
\bes
\left\{\begin{array}{ll}
	-d\Delta S+(b-a+cS_1+cS_2+\ell P_1)S=-\ell S_2P,\;\;&x\in(0,\,l),\\[1mm]
	-D\Delta P+(\rho-\theta S_1)P=\theta P_2S,\;\;&x\in(0,\,l),\\[1mm]
	S=P=0,\;\;&x=0,\,l.
\end{array}\right.
\lbl{4.14a}\ees
We shall prove $S=P=0$. Assume on the contrary that either $S\not=0$ or $P\not=0$. Then $S\not=0$ and $P\not=0$. In deed, $P=0$ implies $S=0$, and $S=0$ implies $P=0$ by \qq{4.14a} since $\ell S_2>0$ and $\theta P_2>0$.

We first show that $P$ must change signs. In fact, as both $(S_1, P_1)$ and $(S_2, P_2)$ are positive solutions of \qq{1.6}, we have that
\bess
\lm_1^d(b-a+cS_1+cS_2+\ell P_1)&>&\lm_1^d(b-a+cS_1+\ell P_1)=0,\\
\lm_1^D(\rho-\theta S_1)&=&0.\eess
Hence, the strong maximum principle holds for $S$. If $P\le 0$, then $S>0$ in $(0,\,l)$ by the strong maximum principle as $P\not=0$. According to $S, P\in X$, we have that $S, P\in C^1([0,\,l])$, and then $S_x(0)>0$, $S_x(l)<0$ by the Hopf boundary lemma. Thus, there exists $\varepsilon>0$ such that $S>-\varepsilon P$ in $(0,\,l)$ (cf. \cite[Lemma 2.2.1]{WangEllip}). It follows from the second equation of \qq{4.14a} that
\bess
\left\{\begin{array}{lll}
	-D\Delta P+(\rho-\theta S_1)P=\theta  P_2 S>-\varepsilon\theta P_2 P,\;\; &x\in\Omega, \\[1mm]
	P=0,\;\; &x\in\partial\Omega.
\end{array}\right.
\eess
As
\[\lm_1^D(\rho-\theta S_1+\varepsilon\theta P_2)>\lm_1^D(\rho-\theta S^*)=0,\]
the strong maximum principle shows that $P>0$ in $(0,\,l)$. This contradicts with the assumption that $P\le 0$. Similarly, $P\le 0$ is also impossible. So, $P$ must change signs.

For the open interval ${\cal I}\subset(0,\,l)$, we denote
 \bess
({\cal L}_1, {\cal I})&=&-d\Delta +b-a+cS_1+cS_2+\ell P_1,\\[1mm]
({\cal L}_2, {\cal I})&=&-D\Delta+(\rho-\theta S_1),\\[1mm]
{\cal D}({\cal L}_i, {\cal I})&=&H^2({\cal I})\cap H_0^1({\cal I}),
\eess
and let $\lm_1({\cal L}_i, {\cal I})$ be the principal eigenvalue of the following eigenvalue problem
\bess\left\{\begin{array}{ll}
	{\cal L}_i\phi=\lm\phi,\;\;&x\in {\cal I},\\[1mm]
	\phi=0,\;\;&x\in\partial {\cal I},
\end{array}\right.
\eess
$i=1,2$. Noticing that $(S_1, P_1)$ is a positive solution of \qq{1.6} and $S_2>0$. We have that
\bess
&\lm_1({\cal L}_1,(0,\,l))>\lm_1^d(b-a+cS_1+\ell P_1)=0,&\\
&\lm_1({\cal L}_2,(0,\,l))=\lm_1^D(\rho-\theta S_1)=0,&\\
&\lm_1({\cal L}_i,{\cal I})>\lm_1({\cal L}_i,(0,\,l))\;\;\text{if}\;\;{\cal I}\varsubsetneqq(0,\,l),\;\;i=1,2.&
\eess
Therefore, the strong maximum principle holds for $({\cal L}_1, {\cal I})$ when ${\cal I}\subset(0,\,l)$, and for $({\cal L}_2, {\cal I})$ when ${\cal I}\varsubsetneqq(0,\,l)$.

Take advantage of the above facts, similar to the proof of \cite[Lemmas 3.3 and 3.4]{LP} we can prove the following claims.

{\bf Claim 1}: The set of zeros of $P$ is discrete in $[0,\,l]$. Moreover, if $w(x)=0$ for some $x\in(0,l)$, then $w$ must change signs in the neighborhood of $x$.

{\bf Claim 2}: Let $[z_1,z_2]\varsubsetneqq[0,\,l]$ with $z_1<z_2$, and  $P(z_1)=P(z_2)=0$. If $P>0$ in $(z_1, z_2)$ and $S(z_1)\le 0$, then $S(z_2)>0$; If $P<0$ in $(z_1, z_2)$ and $S(z_1)\ge 0$, then $S(z_2)<0$.

According to the above Claim 1 we see that the set of zeros of $P$ where it changes sign is discrete in $[0,\,l]$. Let $0=z_0<z_1<\cdots<z_n=l$ be such a set. Then $\lm_1({\cal L}_1, (z_{j-1}, z_j))>0$, $\lm_1({\cal L}_2, (z_{j-1}, z_j))>0$ for $j=1,\cdots, n$. Using the above Claim 2 and repeating the proof of \cite[Theorem 3.1]{LP} we can derive a contradiction that either $S(l)>0$ or $S(l)<0$. This contradiction implies that $S=P=0$.

{\it Step 2: The non-degeneracy}. We have proved in Step 1 that the problem \qq{1.6} has a unique positive solution $(\hat S,\hat P)$. Then \qq{4.14a} becomes exactly the following problem
\bes
\left\{\begin{array}{lll}
	-d\Delta S+(b-a+2c\hat S+\ell\hat P)S=-\ell\hat S P,\;\; &x\in(0,l), \\[1mm]
	-D\Delta P+(\rho-\theta\hat S)P=\theta\hat P S,\;\; &x\in(0,l), \\[1mm]
    S=P=0,\;\;&x=0,\,l.
\end{array}\right.
\lbl{4.15a}\ees
Moreover, we have shown that the problem \qq{4.14a} only has the zero solution in Step 1,  so does \eqref{4.15a}. This shows that the solution $(\hat S, \hat P)$ is non-degenerate.
\end{proof}

\subsection{The existence and uniqueness of positive solutions of the problem \qq{1.3}}

In the above we have known that the problem \qq{1.4} has a positive solution $S^*$ if and only if $\lm_1^d(b-a)<0$, and $S^*$ is unique and non-degenerate when it exists; the problem \qq{1.5} has a positive solution $(\tilde S, \tilde I)$ if and only if $\lm_1^d(b-(k-c)S^*)<0$, and $(\tilde S, \tilde I)$ is unique and non-degenerate when it exists (Theorem \ref{th4.2}); the problem \qq{1.6} has a positive solution $(\hat S,\hat P)$ if and only if $\lm_1^D(\rho-\theta S^*)<0$, and $(\hat S,\hat P)$ is unique and non-degenerate when it exists and $N=1$ (Theorem \ref{th3.4}).

The following theorem concerns with positive solutions of the problem \qq{1.3}.

\begin{theo}\lbl{th3.5} Let $N=1$. Then the problem \qq{1.3} has a positive solution if and only if
 \bes
 \lambda^d_1(b-a)<0,\;\;\lm_1^d(b-(k-c)\hat S+\ell\hat P)<0, \;\;\lm_1^D(\rho-\theta S^*)<0.
 \lbl{4.16}\ees
Moreover, the positive solution $(S, I, P)$ is unique when it exists, and $S+I=\hat S$, $P=\hat P$, where $(\hat S, \hat P)$ is the unique positive solution of \qq{1.6}.
\end{theo}

\begin{proof} Let's prove the necessity and uniqueness first. Let $(S, I, P)$ be a positive solution of \qq{1.3}. Set $u=S+I$. After some straightforward manipulations it can be easily seen that $(u, P)$ satisfies \qq{1.6} with $(S, P)$ replaced by $(u, P)$. Thus $\lm_1^d(b-a)<0$ and $S+I=u<S^*$. Moreover, $(u, P)=(\hat S, \hat P)$ by the uniqueness of positive solutions of \qq{1.6}. Hence, $S<\hat S$. It is easy to see that $k>c$ by the equation of $I$. Thus, by the equations  of $I$ and $P$ in \qq{1.3}, we have that
\bess
 \left\{\begin{array}{lll}
 -d\Delta I=(k-c)SI-bI-cI^2-\ell\hat P I<[(k-c)\hat S-b-\ell \hat P]I,\ &x\in\Omega,\\[1mm]
 -D\Delta P=\theta (S+I)P-\rho P<(\theta S^*-\rho)P,\ &x\in\Omega,\\[1mm]
  I=P=0, \ &x\in\partial\Omega.
 \end{array}\right.
 \eess
Since $I, P>0$, it follows that the last two inequalities in \qq{4.16} hold.

Let $(S_i, I_i, P_i)$, $i=1,2$, be two positive solutions of \qq{1.3}. The above analysis shows that $S_i+I_i=\hat S$ and $P_i=\hat P$. Thus, $I_i$ satisfies
   \begin{eqnarray*}
\left\{\begin{array}{ll}
  -d\Delta I_i=[(k-c)\hat S-b-\ell\hat P] I_i-k I_i^2,\ &x\in\Omega,\\[1mm]
  I_i=0,\;\;&x\in\partial\Omega.
 \end{array}\right.
 \end{eqnarray*}
It is well known that this problem has at most one positive solution. So, $I_1=I_2$, and then $S_1=S_2$. The uniqueness is proved.

In the following we shall prove the sufficiency. Assume that \qq{4.16} hold.
Take $u=(S, I, P)$, $d_1=d_2=d$, $d_3=D$ and $(f_1, f_2, f_3)$ as in \qq{4.9}
The direct calculation gives
\bess
&&(\mathscr{L}+M)F'(S, I, P)\\[1mm]
&=&\left(\begin{array}{cccc}
	a-b-2cS-(c+k)I-\ell P+M  &  a-(c+k)S  &  -\ell S\\[2mm]
	(k-c)I & (k-c)S-b-2cI-\ell P+M  & -\ell I\\[2mm]
	\theta P & \theta P & \theta(S+I)-\rho+M
\end{array}\right).
\eess

{\it Step 1}. Similar to \S\ref{s3.1.1} we have
 \bes
  {\rm index}_W(F,\boldsymbol{0})=0.
 \lbl{4.17}\ees

{\it Step 2: The calculation of ${\rm index}_W(F, (S^*, 0,0))$}. By the direct calculation,
 \[\overline{W}_{(S^*,0,0)}=X\times K\times K,\;\;S_{(S^*,0,0)}=X\times\{(0,0)\}.\]
Take $u=(S^*, 0,0)$ in problems \qq{4.2} and \qq{4.3}. Notice that $\lm_1^d(b-(k-c)S^*)<\lm_1^d(b-(k-c)\hat S+\ell\hat P)<0$ and $\lm_1^D(\rho-\theta S^*)<0$. Similar to the proof of Lemma \ref{l3.1} we can prove that the problem \qq{4.2} has only the zero solution.

Consider the eigenvalue problem
\begin{eqnarray}
{\cal B}z:=(M-D\Delta)^{-1}(\theta S^*-\rho+M)z=\lambda z,\;\; z\in K.
  \label{4.18}
\end{eqnarray}
Since $\lm_1^D(\rho-\theta S^*)<0$, we see that $r\left({\cal B}\right)>1$ is an eigenvalue of \qq{4.18}. Let $\phi_3>0$ be the eigenfunction corresponding to $r\left({\cal B}\right)$.
Similar to the proof of Lemma \ref{l3.1}, we can prove that the problem
  \bess
  (M-d\Delta)^{-1}\big[(a-b-2cS^*+M)\phi_1-\ell S^*\phi_3\big]=r\left({\cal B}\right)\phi_1,\;\;\phi_1\in X\eess
has a unique solution $\phi_1$. Therefore, $r\left({\cal B}\right)>1$ is an eigenvalue of \qq{4.3} and $\phi=(\phi_1, 0, \phi_3)\in\overline{W}_{(S^*,0,0)}\setminus S_{(S^*,0,0)}$ is the corresponding eigenfunction. So, by Corollary \ref{c4.1},
 \bes
 {\rm index}_W(F, (S^*, 0,0))=0.
 \lbl{4.19}\ees

{\it Step 3: The calculation of ${\rm index}_W(F, (\tilde S, \tilde I,0))$}. Similar to \S\ref{s3.1.1} we have
 \[\overline{W}_{(\tilde S, \tilde I,0)}=X\times X\times K,\;\;S_{(\tilde S, \tilde I,0)}=X\times X\times\{0\}.\]
Take $u=(\tilde S, \tilde I,0)$ in problems \qq{4.2} and \qq{4.3}. Let
$\phi=(\phi_1,\phi_2,\phi_3)\in\overline{W}_{(\tilde S, \tilde I,0)}$ be a solution of \qq{4.2}. Noticing that $\tilde S+\tilde I=S^*$ and $\lm_1^D(\rho-\theta S^*)<0$, similar to the proof of Lemma \ref{l3.1} we have $\phi_3=0$, and then $(\phi_1,\phi_2)$ is a solution of the problem \qq{4.12}. In Theorem \ref{th4.2} we have shown that \qq{4.12} has only the zero solution. So, $\phi_1=\phi_2=0$, and then the problem \qq{4.2} has only the zero solution.

It has been known that $r:=r\left({\cal B}\right)>1$ is an eigenvalue of \qq{4.18}. Let
 \bess
 \mathscr{K}=\frac 1r(M-d\Delta)^{-1}
 \left(\begin{array}{lc}M-b+a-2c\tilde S-(k+c)\tilde I&a-(k+c)\tilde S\\[1mm]
 (k-c)\tilde I&M-b+(k-c)\tilde S-2c\tilde I
 \end{array}\right).\eess
Then $\mathscr{K}$ is a compact operator. We first prove that the problem
\bess
  \left(\begin{array}{lll}
 \phi_1\\
 \phi_2\end{array}\right)-\mathscr{K}\left(\begin{array}{lll}
 \phi_1\\
 \phi_2\end{array}\right)=0
 \eess
has only the zero solution in $[H_0^1(\oo)]^2$. In fact, if $(\phi_1, \phi_2)\in [H_0^1(\oo)]^2$ is a solution of \qq{4.20}, then $\phi_i\in C^2(\oo)\cap C^1(\bar\oo)$ by the regularity theory. Thus, $(\phi_1, \phi_2)$ satisfies
 \bes
 \left\{\begin{array}{lll}
 -d\Delta\phi_1+M\phi_1=r^{-1}\big[(M-b+a-2c\tilde S-(k+c)\tilde I)\phi_1+(a-(k+c)\tilde S)\phi_2\big],\; &x\in\Omega, \\[2mm]
 -d\Delta\phi_2+M\phi_2=r^{-1}\big[(k-c)\tilde I\phi_1+(M-b+(k-c)\tilde S-2c\tilde I)\phi_2\big],\; &x\in\Omega, \\[1.5mm]
\phi_1=\phi_2=0,&x\in\partial\Omega.
 \end{array}\right.
 \lbl{4.20}\ees
Let $u=\phi_1+\phi_2$. Then $u$ satisfies
 \bess
 \left\{\begin{array}{lll}
 -d\Delta u+\dd\left[M-r^{-1}(M-b+a-2c S^*)\right]u=0,\;\; &x\in\Omega, \\[1.5mm]
 u=0,&x\in\partial\Omega.
 \end{array}\right.
 \eess
Since $r>1$ and $M-b+a-2c S^*>0$ in $\oo$, by the monotonicity of $\lm_1^d(q)$ in $q$ we have
 \bess
 \lm_1^d\big([M-r^{-1}(M-b+a-2c S^*)]\big)>\lm_1^d(b-a+2c S^*)>0.\eess
The maximum principle implies $u=0$, i.e., $\phi_1=-\phi_2$. Then, by the second equation of \qq{4.20}, we have
 \bess
 \left\{\begin{array}{lll}
  -d\Delta\phi_2+M\phi_2+r^{-1}\big[(b-(k-c)\tilde S+2c\tilde I-M)\phi_2+(k-c)\tilde I\phi_2\big]=0,\;\; &x\in\Omega, \\[1.5mm]
 \phi_2=0,&x\in\partial\Omega.
 \end{array}\right.\qquad
 \eess
Thanks to $\tilde S+\tilde I=S^*$, it follows that
 $$b-(k-c)\tilde S+2c\tilde I-M+(k-c)\tilde I=b-(k-c)S^*+2k\tilde I-M.$$
Hence, $\phi_2$ satisfies
 \bess
 \left\{\begin{array}{lll}
  -d\Delta\phi_2+M\phi_2+r^{-1}\big[b-(k-c)S^*+2k\tilde I-M\big]\phi_2=0,\;\; &x\in\Omega, \\[1.5mm]
 \phi_2=0,&x\in\partial\Omega.
 \end{array}\right.\qquad
 \eess
Similar to the above,
 \bess
 \lm_1^d\big([M+r^{-1}(b-(k-c)S^*+2k\tilde I-M)]\big)>\lambda_1^d(b-(k-c)S^*+k\tilde{I})=0.\eess
The maximum principle asserts $\phi_2=0$. Hence the problem \qq{4.20} has only the zero solution. Let $\phi_3>0$ be the eigenfunction corresponding to $r\left({\cal B}\right)$. Making use of the Fredholm alternative theorem we have that the problem
  \bess
  \left(\begin{array}{lll}
 \phi_1\\
 \phi_2\end{array}\right)-\mathscr{K}\left(\begin{array}{lll}
 \phi_1\\
 \phi_2\end{array}\right)=-\frac 1{r\left({\cal B}\right)}(M-d\Delta)^{-1}\left(\begin{array}{lll}
 \ell\tilde S\phi_3\\
 \ell\tilde I\phi_3\end{array}\right)
 \eess
has a unique solution $(\phi_1, \phi_2)\in[H_0^1(\oo)]^2$. And then $\phi_i\in C^2(\oo)\cap C^1(\bar\oo)$ by the regularity theory. It is easy to verify that $r\left({\cal B}\right)>1$ is an eigenvalue of \qq{4.3} and $\phi=(\phi_1, \phi_2, \phi_3)\in\overline{W}_{(\tilde S, \tilde I,0)}\setminus S_{(\tilde S, \tilde I,0)}$ is the corresponding eigenfunction. Consequently, by Corollary \ref{c4.1},
 \bes
 {\rm index}_W(F, (\tilde S, \tilde I,0))=0.
 \lbl{4.21}\ees

{\it Step 4: The calculation of ${\rm index}_W(F, (\hat S,0,\hat P))$}.
Similar to the above
 \[\overline{W}_{(\hat S,0,\hat P)}=X\times K\times X,\;\;
  S_{(\hat S,0,\hat P)}=X\times\{0\}\times X.\]
Take $u=(\hat S,0,\hat P)$ in problems \qq{4.2} and \qq{4.3}. Noticing that $\lm_1^d(b-(k-c)\hat S+\ell\hat P)<0$ and the problem \qq{4.15a}
has only the zero solution (Theorem \ref{th3.4}). Similar to Step 3 we can show that the problem \qq{4.2} has only the zero solution, and the eigenvalue problem \qq{4.3} has an eigenvalue $\lm>1$. Using Corollary \ref{c4.1} we have
 \bes
 {\rm index}_W(F, (\hat S,0,\hat P))=0.
 \lbl{4.22}\ees

{\it Step 5}. We have known that ${\rm deg}_W(I-F,{\mathcal O})=1$ by \qq{4.7}, ${\rm index}_W(F,\boldsymbol{0})=0$ by \qq{4.17}, ${\rm index}_W(F, (S^*, 0,0))=0$ by
 \qq{4.19}, ${\rm index}_W(F, (\tilde S, \tilde I,0))=0$ by \qq{4.21} and ${\rm index}_W(F, (\hat S,0,\hat P))=0$ by \qq{4.22}. As
  \bess
  &&{\rm index}_W(F,\boldsymbol{0})+{\rm index}_W(F, (S^*, 0,0))+{\rm index}_W(F, (\tilde S, \tilde I,0))+{\rm index}_W(F, (\hat S,0,\hat P))\\
  &\not=&{\rm deg}_W(I-F,{\mathcal O}),\eess
the problem \eqref{1.3} has at least one positive solution.
\end{proof}

\subsection{Stabilities of non-negative trivial equilibrium solutions of \qq{1.2}}

In the above subsections we obtained the nonnegative solutions of \qq{1.3}:
 \bess
 (0,0,0),\;\;(S^*,0,0),\;\;(\hat S,0,\hat P),\;\;(\tilde S, \tilde I,0),\;\;(S_*, I_*, P_*). \eess
They correspond to the nonnegative constant equilibrium solutions of \qq{3.1}:
 \bess
 &\dd(0,0,0),\;\;\kk(\frac{a-b}c,0,0\rr),\;\;\kk(\frac\rho\theta,\,0,\,\frac {a-b-c\rho/\theta}\ell\rr),&\\[2mm]
 &\dd\kk(\frac ak,\,\frac{a(k-c)-bk}{kc},\,0\rr),\;\; \kk(\frac ak,\,\frac{k\rho-a\theta}{k\theta},\,\frac{a-b-c\rho/\theta}\ell\rr).&
  \eess
In the following we study the global asymptotic stabilities of $(0,0,0)$, $(S^*,0,0)$ and $(\tilde S, \tilde I,0)$.

\begin{theo}\lbl{t3.6} Assume that $(\gamma, \sigma)=(\ell,\theta)$. Let $(S,I,P)$ be the unique solution of \qq{1.2}.

{\rm (i)}\, If $\lm^d_0(b-a)\ge 0$, then
 \bess
 \lim_{t\to\infty}(S,I,P)=(0,0,0) \;\;\;\text{in}\;\;C^2(\bar\oo).
  \eess

{\rm(ii)} Assume that $\lm^d_0(b-a)<0$ and $\lm_1^D(\rho-\theta S^*)>0$. If either $k\le c$, or $k>c$ and $\lm_1^d(b-(k-c)S^*)>0$, then
 \bess
 \lim_{t\to\infty}(S,I,P)=(S^*,0,0) \;\;\;\text{in}\;\;C^2(\bar\oo),
  \eess
where $S^*$ is the unique positive solution of \qq{1.4}.

{\rm(iii)} Assume that $\lm^d_0(b-a)<0$. If $\lm_1^d(b-(k-c)S^*)<0$ {\rm(}this implies $k>c${\rm)} and $\lm_1^D(\rho-\theta S^*)>0$, then
 \bess
 \lim_{t\to\infty}(S,I,P)
 =(\tilde S,\tilde I,0) \;\;\;\text{in}\;\;C^2(\bar\oo),
  \eess
where $(\tilde S, \tilde I)$ is the unique positive solution of \qq{1.5}.
 \end{theo}

 \begin{proof} (i) Let $u=S+I$, then we have
  \bes\left\{\begin{array}{ll}
 u_t-d\Delta u=(a-b-cu)u-\ell uP,\;\;&x\in\oo,\;t>0,\\[1mm]
 P_t-D\Delta P=(\theta u-\rho) P,\;\;&x\in\oo,\;t>0,\\[1mm]
 u=P=0, \ &x\in\partial\Omega, \ t>0,\\[1mm]
 u(x,0)=S_0(x)+I_0(x)> 0,\ P(x,0)> 0,\ &x\in\Omega.
\end{array}\right.\label{4.23}
 \ees
Since $\lm^d_0(b-a)\ge 0$, the solution $z$ of the problem
 \bess
 \left\{\begin{array}{lll}
z_t-d\Delta z=(a-b)z-c z^2,\;\; &x\in\Omega, \\[1mm]
 z=0, \ &x\in\partial\Omega,\\[1mm]
 z(x,0)=u(x,0)> 0,\ &x\in\Omega
 \end{array}\right.
 \eess
satisfies $\lim_{t\to\infty} z=0$ in $C^2(\boo)$. And then $\lim_{t\to\infty} u=0$ in $C(\boo)$ by the comparison principle, which implies that $\lim_{t\to\infty} S=\lim_{t\to\infty} I=0$ in $C(\boo)$ as $S, I>0$. From the equation of $P$ we have that $\lim_{t\to\infty} P=0$ in $C(\boo)$. Recalling the estimate \qq{2.3}, it is easy to see that these limits hold in $C^2(\boo)$.\sk

(ii)\, We first consider the problem \qq{4.23}. Rewrite the equation of $P$ as
 \bes
 P_t=D\Delta P-(\rho-\theta S^*)P+\theta(u-S^*) P.
 \lbl{4.24}\ees
Let $\phi$ be the corresponding positive eigenfunction to $\lm_1^D(\rho-\theta S^*)$. Multiplying \qq{4.24} by $\phi$ and integrating by parts we have that, through detailed calculation,
 \bess
 \frac{{\rm d}}{{\rm d}t}\int_\oo P\phi{\rm d}x&=&\int_\oo[\phi D\Delta P-(\rho-\theta S^*)P\phi+\theta(u-S^*) P\phi]\dx\\[.5mm]
 &=&\int_\oo[P D\Delta\phi-(\rho-\theta S^*)P\phi+\theta(u-S^*) P\phi]\dx\\[.5mm]
  &=&\int_\oo\left[-\lm_1^D(\rho-\theta S^*)+\theta(u-S^*)\right]P\phi\dx.
 \eess
As $u$ satisfies $u_t-\Delta u<(a-b-cu)u$, it follows that $\limsup_{t\to\infty}u\le S^*$ uniformly in $\boo$. On account of $-\lm_1^D(\rho-\theta S^*)<0$, there exist $\varepsilon>0$ and $T\gg 1$ such that
 \[\frac{{\rm d}}{{\rm d}t}\int_\oo P\phi{\rm d}x\le -\varepsilon\int_\oo P\phi{\rm d}x,\;\; t\ge T,\]
which implies that
 \[\lim_{t\to\infty}\int_\oo P\phi{\rm d}x=0.\]
On the basis of the estimate \qq{2.3} we still have that
 \bes
 \lim_{t\to\infty} P=0\;\;\;\text{in}\;\;C^2(\boo).
 \lbl{4.25}\ees
Applying the comparison arguments to the first equation of \qq{4.23} and using the estimate \qq{2.3}, we can derive that
\bes
 \lim_{t\to\infty}(S+I)=\lim_{t\to\infty} u=S^*\;\;\;\text{in}\;\;C^2(\boo).
 \lbl{4.26}\ees

For the case $k\le c$. It is clear from the equation of $I$ in \qq{1.2} that $\lim_{t\to\infty} I=0$ in $C(\boo)$, and so in $C^2(\boo)$ by using the estimate \qq{2.3}.
 Therefore, $\lim_{t\to\infty} S=S^*$ in $C^2(\boo)$ by \qq{4.26}.

For the case that $k>c$ and $\lm_1^d(b-(k-c)S^*)>0$.
Similar to the above, we write the second equation of \qq{1.2} as
\bes
 I_t&=&d\Delta I+[(k-c)(S-S^*)-cI-\gamma P]I-[b-(k-c)S^*]I\nonumber\\
 &\le&d\Delta I-[b-(k-c)S^*]I+(k-c)(S-S^*)I.
 \lbl{4.27}\ees
Let $\varphi$ be the corresponding positive eigenfunction to $\lm_1^d(b-(k-c)S^*)$. Multiplying \qq{4.27} by $\phi$ and integrating by parts we have
 \bess
 \frac{{\rm d}}{{\rm d}t}\int_\oo I\varphi{\rm d}x&\le&
 \int_\oo\big\{\varphi d\Delta I-[b-(k-c)S^*]I\varphi\big\}\dx+\int_\oo(k-c)(S-S^*)I\varphi\dx\\
  &=&\int_\oo I\big \{d\Delta\varphi-[b-(k-c)S^*]\varphi\big\}\dx+\int_\oo(k-c)(S-S^*)I\varphi\dx\\
 &=&\int_\oo[-\lm_1^d(b-(k-c)S^*)+(k-c)(S-S^*)]I\varphi\dx.
 \eess
We have known that $\limsup_{t\to\infty}S\le S^*$ uniformly in $\boo$.
Similar to the above we can derive that $\lim_{t\to\infty} I=0$ in $C^2(\boo)$. So $\lim_{t\to\infty} S=S^*$ in $C^2(\boo)$ by \qq{4.26}.\sk

(iii) We first consider the problem \qq{4.23}. In part (ii) we have obtained \qq{4.25} and \qq{4.26}. Rewriting the equation of $I$ as
 \bess\left\{\begin{array}{ll}
 I_t=d\Delta I+[(k-c)S^*-b+(k+c)(u-S^*)-\ell P]I-kI^2,\;\;&x\in\oo,\;t>0,\\[1mm]
 I=0, \ &x\in\partial\Omega, \ t>0,\\[1mm]
 I(x,0)=S_0(x)+I(x,0)> 0.\ &x\in\Omega.
\end{array}\right.
 \eess

Since $\lm^d_0(b-a)<0$ and $\lm_1^d(b-(k-c)S^*)<0$, by Theorem \ref{th4.2} we know that \qq{1.5} has a unique positive solution $(\tilde S, \tilde I)$, and $\tilde I$ satisfies
\bess\left\{\begin{array}{ll}
 -d\Delta \tilde I=[(k-c)S^*-b]\tilde I-k\tilde I^2,\;\;&x\in\oo,\\[1mm]
 \tilde I=0, \ &x\in\partial\Omega.
\end{array}\right.
 \eess
Recall that
 \bess
 \lim_{t\to\infty}[(k+c)(u-S^*)-\ell P]=0\;\;\;\text{in}\;\;C^2(\boo).
  \eess
Make use of the comparison arguments it can be shown that $\lim_{t\to\infty} I=\tilde I$ in $C^2(\boo)$. In turn, $\lim_{t\to\infty} S=S^*-\tilde I=\tilde S$ in $C^2(\boo)$ by \qq{4.26} and Theorem \ref{th4.2}.
\end{proof}

\vskip 4pt \noindent {\bf Acknowledgment:} The author would like to thank Professors Y. H. Du and Y. Lou for their suggestions on the uniqueness of positive solutions of \qq{1.6}, and provided me the reference \cite{LP}. The author would also like to thank Dr. H. M. Huang, who helped me check and revise this manuscript many times.

\end{document}